\documentclass[12pt,fleqn,leqno]{amsart}

\usepackage{amsfonts,amssymb}
\usepackage{enumitem}
\usepackage[svgnames]{xcolor}
\usepackage[colorlinks,linkcolor=blue,citecolor=Green]{hyperref}
\usepackage{titlesec}
\usepackage{fourier}
\usepackage[nameinlink]{cleveref}
\usepackage{centernot}

\usepackage[a4paper,centering]{geometry}

\titleformat{\section}[block]
 {\bfseries}
 {\thesection.}
 {\fontdimen2\font}
 {}

\setenumerate{label=\upshape{(\alph*)}}

\newtheorem{theorem}{Theorem}[section]
\newtheorem{proposition}[theorem]{Proposition}

\theoremstyle{definition}
\newtheorem{remark}[theorem]{Remark}

\DeclareMathOperator{\R}{\mathbb{R}}
\DeclareMathOperator{\uhr}{\upharpoonright}
\DeclareMathOperator{\st}{st}
\DeclareMathOperator\sto{\leadsto}
\DeclareMathOperator\conv{conv}
\DeclareMathOperator\coz{coz}
\DeclareMathOperator\car{car}
\DeclareMathOperator\supp{supp}

\renewcommand{\emptyset}{\varnothing}
\numberwithin{equation}{section}

\begin{document}

\author{Valentin Gutev} \address{Institute of Mathematics and
  Informatics, Bulgarian Academy of Sciences, Acad. G. Bonchev Street,
  Block 8, 1113 Sofia, Bulgaria}

\email{\href{mailto:gutev@math.bas.bg}{gutev@math.bas.bg}}



\title{Continuous Selections, Function Spaces and Partitions of Unity}

\begin{abstract}
  The famous Michael selection theorem deals with the characterisation
  of paracompact spaces by continuous selections of lower
  semi-continuous mappings in Banach spaces. In this paper, we will
  discuss several equivalent forms of this theorem, without explicitly
  mentioning paracompactness. This will be based on a previous result,
  also obtained by Michael, that a space $X$ is paracompact if and
  only if every open cover of $X$ has an index-subordinated partition
  of unity. Thus, we will show that the existence of such partitions
  of unity on a space $X$ is equivalent to the existence of continuous
  selections for special lower semi-continuous mappings from $X$ to
  the nonempty convex subsets of special function spaces.
\end{abstract}

\date{\today}
% \date{}
\maketitle

\section{Introduction}

All topological spaces in this paper are assumed to be Hausdorff. For
spaces (sets) $X$ and $Y$, we will write $\Phi:X\sto Y$ to designate
that $\Phi$ is a map from $X$ to the \emph{nonempty} subsets of $Y$,
i.e.\ that $\Phi$ is a \emph{set-valued mapping} (called also a
\emph{multifunction}, or simply a \emph{carrier}
\cite{michael:56a}). A mapping $\Phi:X\sto Y$ is \emph{lower
  semi-continuous}, or \emph{l.s.c.}, if the set
$\Phi^{-1}[U]=\left\{x\in X:\Phi(x)\cap U\neq \emptyset\right\}$ is
open in $X$ for every open $U\subseteq Y$.  Finally, let us also
recall that a usual map $f:X\to Y$ is a \emph{selection} (or a
\emph{single-valued selection}) for $\Phi:X\sto Y$ if $f(x)\in\Phi(x)$
for every $x\in X$.\medskip

In 1956, Ernest Michael obtained his famous selection theorem
\cite[Theorem 3.2$''$]{michael:56a} that a space $X$ is paracompact if
and only if for every Banach space $E$, every closed-convex-valued
l.s.c.\ mapping $\Phi:X\sto E$ has a continuous selection. In this
paper, the definition of paracompactness or properties of paracompact
spaces will not play any role. However, for the proper understanding
of our results, the reader should keep in mind that a space $X$ is
paracompact if and only if every open cover of $X$ has an
index-subordinated partition of unity\,---\,a result previously
obtained by Michael in \cite[Proposition 2]{michael:53}. Partitions of
unity are discussed in \Cref{sec:funct-open-relat}, which also
contains the corresponding definitions. In fact, all our
considerations will be based on the idea that partitions of unity are
nothing else but continuous maps in a certain subset of
$\ell_1(\mathcal{A})$, where $\ell_1(\mathcal{A})$ is the standard
Banach space of all functions $y: \mathcal{A}\to \R$ satisfying the
condition $\sum_{\alpha\in \mathcal{A}}|y(\alpha)| < +\infty$, its
linear operations are pointwise defined and the norm is
$\|y\|_1 = \sum_{\alpha\in \mathcal{A}}|y(\alpha)|$,
$y\in \ell_1(\mathcal{A})$. This interpretation is well known and is
pointed out in \Cref{theorem-Simpl-Qst-v22:1}. It is also the basis of
Michael's proof of \cite[Theorem 3.2$''$]{michael:56a} and the
alternative proof of this theorem given by Kand\^o in
\cite{MR0063648}, but the connection has somehow gone a little
unnoticed. In this paper, see \Cref{proposition-shsa-vgg-rev:1} and
\Cref{theorem-CS-and-CM-16:2}, we will give a simple proof of the
following result, which was actually obtained by Michael in
\cite{michael:56a}.

\begin{theorem}[\cite{michael:56a}]
  \label{theorem-CS-and-CM-19:1}
  For a space $X$ and an infinite set $\mathcal{A}$, the following
  conditions are equivalent\textup{:}
  \begin{enumerate}
  \item\label{item:CS-and-CM-22:4} Each closed-convex-valued l.s.c.\
    mapping\/ $\Phi:X\sto \ell_1(\mathcal{A})$ has a continuous
    selection.
  \item\label{item:CS-and-CM-22:5} Each open cover\/
    $\left\{U_\alpha:\alpha\in \mathcal{A}\right\}$ of $X$ has an
    index-subordinated partition of unity.
  \end{enumerate}
\end{theorem}

Regarding the role of the set $\mathcal{A}$, let us remark that in
\cite[Theorem 3.1$''$]{michael:56a}, Michael showed that a space $X$
is normal and countably paracompact if and only if every
closed-convex-valued l.s.c.\ mapping $\Phi:X\sto \R$ has a continuous
selection. Accordingly, \ref{item:CS-and-CM-22:4} of
\Cref{theorem-CS-and-CM-19:1} represents a characterisation of
countably paracompact normal spaces for every finite set
$\mathcal{A}$. In contrast, for a finite set $\mathcal{A}$ with
cardinality $|\mathcal{A}|\geq2$, \ref{item:CS-and-CM-22:5} of this
theorem is only a characterisation of normal spaces.\medskip

For a finite set $\mathcal{A}$, \ref{item:CS-and-CM-22:4} of
\Cref{theorem-CS-and-CM-19:1} is still a characteristic property of
countably paracompact normal spaces if $\ell_1(\mathcal{A})$ is
replaced by the Hilbert space $\ell_2(\mathcal{A})$ of all functions
${y: \mathcal{A}\to \R}$ with
$\sum_{\alpha\in \mathcal{A}}y^2(\alpha) < +\infty$, where the norm is
$\|y\|_2 = \sqrt{\sum_{\alpha\in \mathcal{A}}y^2(\alpha)}$,
$y\in \ell_2(\mathcal{A})$. However, for an infinite set
$\mathcal{A}$, \Cref{theorem-CS-and-CM-19:1} is no longer valid if
$\ell_1(\mathcal{A})$ is replaced by $\ell_2(\mathcal{A})$.  The
following interesting result was obtained by Ivailo Shishkov in
\cite{MR2406397}, and is related to a more general question posed in
\cite{MR1606612}. The interested reader can consult \cite{Gutev2020a},
where this question is discussed in detail.

\begin{theorem}[\cite{MR2406397}]
  \label{theorem-CS-and-CM-19:2}
  For a space $X$ and an infinite set $\mathcal{A}$, the following
  conditions are equivalent\textup{:}
  \begin{enumerate}
  \item\label{item:CS-and-CM-19:1} Each closed-convex-valued l.s.c.\
    mapping\/ $\Phi:X\sto \ell_2(\mathcal{A})$ has a continuous
    selection.
  \item\label{item:CS-and-CM-19:2} For every locally finite cover
    $\left\{F_\alpha:\alpha\in \mathcal{A}\right\}$ of a closed subset
    of $X$ there exists a locally finite open cover
    $\left\{U_\alpha:\alpha\in \mathcal{A}\right\}$ of\/ $X$ such that
    $F_\alpha\subseteq U_\alpha$ for every $\alpha\in \mathcal{A}$.
  \end{enumerate}
\end{theorem}

In \cite{MR2406397}, Shishkov showed that a space $X$ is countably
paracompact and collectionwise normal if and only if it satisfies
\ref{item:CS-and-CM-19:1} of \Cref{theorem-CS-and-CM-19:2} for every
set $\mathcal{A}$.  However, his proof was actually done for a fixed
set $\mathcal{A}$ and it is well known that a space $X$ is countably
paracompact and $|\mathcal{A}|$-collectionwise normal exactly when
it satisfies \ref{item:CS-and-CM-19:2} of
\Cref{theorem-CS-and-CM-19:2}, see \cite{dowker:56,katetov:58} and
also \cite{MR0284966}.\medskip

Each mapping $\Phi:X\sto Y$ can be identified with its \emph{graph} in
$X\times Y$, i.e.\ with the \emph{relation}
$\Phi=\left\{\langle x,y\rangle\in X\times Y: y\in
  \Phi(x)\right\}$. Thus, for spaces $X$ and $Y$, the property that
$\Phi\subseteq X\times Y$ is an \emph{open set} simply means that
$\Phi:X\sto Y$ has an \emph{open graph}. Each open-graph mapping
$\Phi:X\sto Y$ is l.s.c., but there are simple examples of l.s.c.\
mappings which are not open relations, see e.g.\
\cite{Gutev2023}.\medskip

The selection problem for convex-valued open-graph mappings is
naturally related to the function space $C_0(T)$ of all continuous
functions $y:T\to \R$ on a space $T$ such that the set
$\left\{t\in T: |y(t)|\geq \varepsilon\right\}$ is compact for every
$\varepsilon>0$. The linear operations in $C_0(T)$ are defined
pointwise and it is equipped with the sup-norm
$\|y\|=\sup_{t\in T}|y(t)|$, $y\in C_0(T)$. We will use $C_0(T)$ in
the special case when $T=\tau$ is an infinite initial ordinal equipped
with the usual order topology. Finally, let us recall that a cover
$\left\{U_\alpha:\alpha<\tau\right\}$ of a space $X$ indexed by an
ordinal $\tau$ is \emph{increasing} if $U_\alpha\subseteq U_\beta$ for
every $\alpha<\beta<\tau$. The following theorem was obtained in
\cite[Theorem 1.1]{Gutev2023}.

\begin{theorem}[\cite{Gutev2023}]
  \label{theorem-CS-and-CM-19:3}
  For a space $X$ and an infinite cardinal number $\tau$, the
  following conditions are equivalent\textup{:}
  \begin{enumerate}
  \item\label{item:CS-and-CM-19:3} Each convex-valued open-graph
    mapping\/ $\Phi:X\sto C_0(\tau)$ has a continuous selection.
  \item\label{item:CS-and-CM-19:4} $X$ is normal and each increasing
    open cover\/ $\left\{U_\alpha:\alpha<\tau\right\}$ of\/ $X$ has an
    index-subordinated partition of unity.
  \end{enumerate}
\end{theorem}

This result is based on the characterisation of paracompactness using
directed open covers, which was obtained by Mack in \cite[Theorem
5]{MR0211382}. In fact, it was shown in \cite{Gutev2023} that the
selection property in \ref{item:CS-and-CM-19:3} of
\Cref{theorem-CS-and-CM-19:3} is equivalent to normality and
$\tau$-paracompactness of $X$. Here, condition
\ref{item:CS-and-CM-19:4} is formulated in terms of partitions of
unity to emphasise the difference with the function space
$\ell_1(\mathcal{A})$ in \Cref{theorem-CS-and-CM-19:1}. The interested
reader is also referred to \cite[Theorem 5.7]{MR1961298} for another
similar characterisation of $\tau$-paracompactness.\medskip

The maps in the punctured vector space
$\R^\mathcal{A}\setminus\{\mathbf{0}\}$, obtained by removing the
origin $\mathbf{0}\in \R^\mathcal{A}$ from this space, are naturally
related to multivalued transformations representing the
``\emph{$\mathcal{A}$-carrier}'' of such maps, which are simply covers
of the domain indexed by the set $\mathcal{A}$. It is based on the
interpretation that all covers of a set $X$ indexed by the elements of
a set $Y$ are identical to the set-valued mappings $\Phi:X\sto Y$, see
\cite{MR3673071,Gutev2023}. Namely, the \emph{indexed cover} of $X$
represented by a mapping $\Phi:X\sto Y$ is simply the \emph{inverse
  relation} $\Phi^{-1}\subseteq Y\times X$ defined by
$Y\ni y\to \Phi^{-1}(y)=\left\{x\in X: y\in \Phi(x)\right\}\subseteq
X$.\medskip

For a space $X$, topological properties of such covers do not depend
on the indexing set $Y$. In this regard, let us remark that every
open-graph mapping $\Phi:X\sto Y$ is an \emph{open cover} of $X$, but
there are simple examples of open covers $\Phi:X\sto Y$ that are not
open relations in $X\times Y$. In fact, one can easily see that
$\Phi:X\sto Y$ is an open cover of $X$ exactly when it is an l.s.c.\
mapping with respect to the discrete topology on $Y$ (equivalently,
with respect to any topology on $Y$). Such mappings $\Phi:X\sto Y$
have the property that $\Phi^{-1}[U]$ is open in $X$ for every
$U\subseteq Y$ and we will call them \emph{totally-l.s.c.} Thus, in
our terminology, the open covers of $X$ indexed by a set $Y$ are
identical to the totally-l.s.c.\ mappings $\Phi:X\sto Y$. To emphasise
the connection and difference with \Cref{theorem-CS-and-CM-19:1}, we
will use the synonym ``totally-l.s.c.''. In some technical parts of
this paper, however, it will be more practical to rely on the
open-cover interpretation.\medskip

A function $\xi:\mathcal{A}\to \R$ is said to have a \emph{finite
  support} if
$\left\{\alpha\in \mathcal{A}: \xi(\alpha)\neq 0\right\}$ is a finite
set.  The collection of all functions $y\in\R^\mathcal{A}$ with a
finite support is denoted by $\mathbf{c}_{00}(\mathcal{A})$ and is a
linear subspace of $\ell_1(\mathcal{A})$. It is a normed space with
respect to the $\|\cdot\|_1$-norm on $\ell_1(\mathcal{A})$ and unless
otherwise stated, we will always consider
$\mathbf{c}_{00}(\mathcal{A})$ endowed with this norm. However,
$\mathbf{c}_{00}(\mathcal{A})$ is not a Banach space, it is not
complete. In this paper, we will also prove the following result which
is complementary to \Cref{theorem-CS-and-CM-19:1}.

\begin{theorem}
  \label{theorem-CS-and-CM-19:5}
  For a space $X$, the following conditions are equivalent\textup{:}
  \begin{enumerate}
  \item\label{item:CS-and-CM-19:5} If\/ $\mathcal{A}$ is a set, then
    each convex-valued totally-l.s.c.\ mapping\/
    $\Phi:X\sto \mathbf{c}_{00}(\mathcal{A})$ has a continuous
    selection.
  \item\label{item:CS-and-CM-19:6} Each open cover of\/ $X$ has an
    index-subordinated partition of unity.
  \end{enumerate}
\end{theorem}

It should be noted that, unlike
\Cref{theorem-CS-and-CM-19:1,theorem-CS-and-CM-19:3}, the above result
is only a characterisation of paracompact spaces, not of normal and
$\tau$-paracompact spaces for some fixed cardinal
$\tau=|\mathcal{A}|$. This is natural because only the linear
structure on the vector space $\mathbf{c}_{00}(\mathcal{A})$ plays a
role in \Cref{theorem-CS-and-CM-19:5}. For the proper understanding of
these theorems, let us also explicitly point out the following diagram
regarding how these variations of l.s.c.\ mappings compare to each
other:
\[
  \text{open-graph}\implies\text{totally-l.s.c.}
  \implies\text{l.s.c.}\centernot\implies \text{totally-l.s.c.}
  \centernot\implies \text{open-graph}
\]

The preparation for the proof of \Cref{theorem-CS-and-CM-19:5} is done
in \Cref{sec:funct-open-relat,sec:univ-locally-finite}. Briefly,
\Cref{sec:funct-open-relat} contains several observations about
functionally open covers of a space $X$. Partitions of unity on a
space $X$ indexed by a set $\mathcal{A}$ are special functionally open
covers of $X$. In fact, as already mentioned, they are continuous maps
in a certain subset
$\mathfrak{S}(\mathcal{A})\subseteq \ell_1(\mathcal{A})$ of the Banach
space $\ell_1(\mathcal{A})$, see \cref{eq:CS-and-CM-13:1} and
\Cref{theorem-Simpl-Qst-v22:1}. The advantage of being able to
consider a partition of unity as a usual continuous map in the subset
$\mathfrak{S}(\mathcal{A})\subseteq\ell_1(\mathcal{A})$ is that we can
now consider compositions of partitions of unity. With this idea in
mind, in \Cref{theorem-CS-and-CM-15:1} of
\Cref{sec:univ-locally-finite} we will construct a special continuous
map from $\mathfrak{S}(\mathcal{A})$ to itself.  It stands for a
``universal'' locally finite partition of unity in the sense that each
partition of unity on a space $X$ indexed by the set $\mathcal{A}$ is
``transformed'' by the composition with this map into a locally finite
partition of unity on $X$ with respect to the same indexed set
$\mathcal{A}$.  The proof of \Cref{theorem-CS-and-CM-19:5} will be
finalised in \Cref{sec:select-part-unity}, in fact it will be obtained
as a consequence of a more general result (see
\Cref{theorem-CS-and-CM-16:1}). In the same section, using approximate
continuous selections, we will refine this result to a similar
characterisation, but now for $\tau$-paracompact normal spaces for a
fixed cardinal $\tau$ (see \Cref{theorem-CS-and-CM-16:2}). This not
only adds another equivalent condition to
\Cref{theorem-CS-and-CM-19:1}, but also simplifies the proof of this
theorem (see \Cref{proposition-shsa-vgg-rev:1}). In the last
\Cref{sec:canon-maps-select} of this paper, we will extend
\Cref{theorem-CS-and-CM-19:5} to the case when
$\mathbf{c}_{00}(\mathcal{A})$ is equipped with the finite topology
(see \Cref{theorem-st-app-v12:1}). This illustrates a natural
relationship between selections continuous with respect to the finite
topology on vector spaces and canonical maps of open covers,
i.e. special continuous maps in the geometric realisation of the
nerves of these covers.

\section{Functionally Open Relations}
\label{sec:funct-open-relat}

Let $\pi_\alpha:\R^\mathcal{A}\to \R$ be the projection onto the
$\alpha$-th factor ($\alpha\in \mathcal{A}$) of the product of
$\mathcal{A}$~copies of the real line $\R$. Every collection of
functions $\xi_\alpha:X\to \R$, $\alpha\in \mathcal{A}$, is identical
to a map $\xi:X\to \R^\mathcal{A}$ and vice versa. Simply put,
$\xi=\Delta_{\alpha\in\mathcal{A}}\xi_\alpha$ is the diagonal product
of the functions $\xi_\alpha$, $\alpha\in \mathcal{A}$, while
$\xi_\alpha=\pi_\alpha\circ \xi$, $\alpha\in \mathcal{A}$, are the
coordinate functions of $\xi$. For a map $\xi:X\to \R^\mathcal{A}$, we
will use $\xi[x]$ for its value at a point $x\in X$ to stress on the
fact that $\xi[x]:\mathcal{A}\to \R$ is a function. In fact, the maps
$\xi:X\to \R^\mathcal{A}$ and the functions
$\xi:X\times \mathcal{A}\to \R$ are the same thing, but this will play
no role in our considerations.\medskip

For a function $\xi:X\to \R$, the set
$\coz(\xi)=\left\{x\in X:\xi(x)\neq 0\right\}$ is called the
\emph{cozero set} or \emph{set-theoretic support} of $\xi$. The
\emph{carrier} $\car(y)\subseteq \mathcal{A}$ of an element
$y\in \R^\mathcal{A}$ is the cozero set of the function
$y:\mathcal{A}\to \R$, i.e.\ $\car(y)=\coz(y)$. Thus, to each map
$\xi:X\to \R^\mathcal{A}$ we will associate the relation
$\car_\xi\subseteq X\times\mathcal{A}$ defined by
\begin{equation}
  \label{eq:Simpl-Qst-v11:1}
  \car_\xi(x)=\car\left(\xi[x]\right)=\coz(\xi[x])\quad \text{for
    every $x\in X$.} 
\end{equation}
The inverse relation
$\st_\xi=\car_\xi^{-1}\subseteq \mathcal{A}\times X$, known as the
\emph{open star} of a map $\xi:X\to \R^\mathcal{A}$, is given simply
by
\begin{equation}
  \label{eq:Simpl-Qst-v16:2}
  {\st_\xi(\alpha)= \coz\left(\xi_\alpha\right)}\quad \text{for
    every $\alpha\in \mathcal{A}$.}
\end{equation}

A subset $U\subseteq X$ of a space $X$ is \emph{functionally open} if
$U=\coz(\xi)$ for some continuous function $\xi:X\to \R$. Functionally
open covers are essentially identical to continuous maps in the
Tychonoff product of real lines.  Indeed, the covers of $X$ indexed by
a set $\mathcal{A}$ are exactly the set-valued mappings from $X$ to
the nonempty subsets of $\mathcal{A}$. Since a map
$\xi:X\to \R^\mathcal{A}$ is continuous if and only if its coordinate
functions ${\xi_\alpha:X\to \R}$, $\alpha\in \mathcal{A}$, are
continuous, we obtain the following simple observation.

\begin{proposition}
  \label{proposition-Simpl-Qst-v24:3}
  A cover $\Omega:X\sto \mathcal{A}$ of a space $X$ is functionally
  open if and only if there exists a continuous map\/
  $\xi:X\to \R^\mathcal{A}\setminus \{\mathbf{0}\}$ such that\/
  $\car_\xi=\Omega$.
\end{proposition}

A collection $\xi_\alpha:X\to [0,1]$, $\alpha\in \mathcal{A}$, of
continuous functions on a space $X$ is a \emph{partition of unity} if
$\sum_{\alpha\in \mathcal{A}}\xi_\alpha(x)=1$, for each $x\in X$.
Here, ``$\sum_{\alpha\in \mathcal{A}}\xi_\alpha(x)=1$'' means that
only countably many functions $\xi_\alpha$'s do not vanish at $x$, and
the series composed by them is convergent to 1. The diagonal map
$\xi=\Delta_{\alpha\in \mathcal{A}}\xi_\alpha:X\to
[0,1]^\mathcal{A}\subseteq \R^\mathcal{A}$ generated by a partition of
unity $\left\{\xi_\alpha:\alpha\in \mathcal{A}\right\}$ on a space $X$
is a map into the \emph{positive cone} $\ell_1^+(\mathcal{A})$ of the
space $\ell_1(\mathcal{A})$, i.e.\ into the set
\[
  \ell_1^+(\mathcal{A})=\left\{y\in \ell_1(\mathcal{A}): y(\alpha)\geq
    0\ \text{for every $\alpha\in \mathcal{A}$}\right\}.
\]
In fact, this map is in the convex subset
$\mathfrak{S}(\mathcal{A})\subseteq \ell_1^+(\mathcal{A})\setminus
\{\mathbf{0}\}$ defined by
\begin{equation}
  \label{eq:CS-and-CM-13:1}
  \mathfrak{S}(\mathcal{A})=\left\{y\in\ell_1^+(\mathcal{A}):
    \|y\|_1=1\right\}. 
\end{equation}

Accordingly, by \Cref{proposition-Simpl-Qst-v24:3}, partitions of
unity are special functionally open covers of a space $X$.  The
following natural result is the implication (f)$~\implies~$(a) of
\cite[Theorem 1.2]{morita:60} and was explicitly stated in
\cite[Proposition 5.4]{MR1425941}; a special case of this result was
also obtained in the proof of \cite[Theorem 1]{MR597065}.

\begin{theorem}
  \label{theorem-Simpl-Qst-v22:1}
  A collection of functions $\xi_\alpha:X\to \R$,
  $\alpha\in \mathcal{A}$, is a partition of unity on a space $X$ if
  and only if the diagonal map
  $\xi=\Delta_{\alpha\in\mathcal{A}}\xi_\alpha$ takes values in
  $ \mathfrak{S}(\mathcal{A})$ and is continuous.
\end{theorem}

In what follows, we will use \Cref{theorem-Simpl-Qst-v22:1} without
any explicit reference. That is, when we refer to a partition of a
unity on a space $X$, we will simply mean and often say that this is a
continuous map $\xi:X\to \mathfrak{S}(\mathcal{A})$ for some set
$\mathcal{A}$.\medskip

An indexed cover $\Omega:X\sto \mathcal{A}$ of a space $X$ is
\emph{locally finite} if every point $p\in X$ is contained in an open
set $U\subseteq X$ such that
$\left\{\alpha\in\mathcal{A}: U\cap \Omega^{-1}(\alpha)\neq
  \emptyset\right\}$ is a finite set or, in other words, if
$\Omega[U]$ is a finite subset of $\mathcal{A}$. In case
${\car_\eta:X\sto \mathcal{A}}$ is a locally finite cover of a space
$X$ for some $\eta:X\to \R^\mathcal{A}\setminus\{\mathbf{0}\}$, then
not only
$\eta:X\to \mathbf{c}_{00}(\mathcal{A})\setminus\{\mathbf{0}\}$, but
also $\eta$ is not as arbitrary as it might seem at first
glance. Namely, we will say that a map $\eta:X\to E$ in a vector space
$E$ is \emph{locally finite-dimensional} if every point $p\in X$ is
contained in an open set $U\subseteq X$ such that $\eta(U)\subseteq L$
for some finite-dimensional subspace $L\subseteq E$. We now have the
following natural interpretation of such covers.

\begin{proposition}
  \label{proposition-CS-and-CM-20:1}
  Let $X$ be a space and\/
  $\eta:X\to \R^\mathcal{A}\setminus\{\mathbf{0}\}$. Then
  $\car_\eta:X\sto \mathcal{A}$ is a locally finite cover of\/ $X$ if
  and only if\/ $\eta$ is a locally finite-dimensional map.
\end{proposition}

\begin{proof}
  The map $\eta:X\to \R^\mathcal{A}\setminus\{\mathbf{0}\}$ is locally
  finite-dimensional if and only if for each point $p\in X$ there
  exists an open set $U\subseteq X$, with $p\in U$, and a finite subset
  $\mathcal{B}\subseteq \mathcal{A}$ such that
  $\eta(U)\subseteq \R^\mathcal{B}\setminus
  \{\mathbf{0}\}$. Obviously, this is equivalent to the fact that
  ${\car_\eta[U]\subseteq \mathcal{B}}$, i.e., to the property that
  $\car_\eta:X\sto \mathcal{A}$ is a locally finite cover of $X$.
\end{proof}

According to
\Cref{proposition-Simpl-Qst-v24:3,proposition-CS-and-CM-20:1}, if
$\Omega:X\sto \mathcal{A}$ is a locally finite functionally open cover
of a space $X$, then there exists a continuous locally
finite-dimensional map
$\eta:X\to \mathbf{c}_{00}(\mathcal{A})\setminus \{\mathbf{0}\}$ such
that $\car_\eta=\Omega$. In fact, in this case, the map $\eta$ is
continuous in a stronger sense. To this end, let us recall that the
\emph{finite topology} on a vector space $E$ is the weak topology
determined by the Euclidean topology on each finite-dimensional linear
subspace of $E$, see \cite{dugundji:66,MR0089373}. In other words, a
subset $U\subseteq E$ is open in the finite topology if and only if
$U\cap L$ is open in $L$ for every finite-dimensional linear subspace
$L\subseteq E$. This topology is always Hausdorff, and it is the
finest locally convex linear topology on $E$ when the linear dimension
of $E$ is countable \cite{Kakutani1963}, see also \cite[Appendix One,
A.4.3]{dugundji:66}. This is no longer true if the dimension of $E$ is
uncountable, in which case the vector addition is continuous in each
variable separately, but is not jointly continuous.

\begin{proposition}
  \label{proposition-Simpl-Qst-v26:1}
  If\/ $\eta:X\to \mathbf{c}_{00}(\mathcal{A})\setminus\{\mathbf{0}\}$
  is a locally finite-dimensional map such that each coordinate
  function $\eta_\alpha:X\to \R$, $\alpha\in \mathcal{A}$, is
  continuous, then $\eta$ is continuous with respect to the finite
  topology on $\mathbf{c}_{00}(\mathcal{A})$.
\end{proposition}

\begin{proof}
  By definition, a point $p\in X$ is contained in an open subset
  ${U\subseteq X}$ such that
  $\eta(U)\subseteq \R^\mathcal{B}\subseteq
  \mathbf{c}_{00}(\mathcal{A})$ for some finite subset
  $\mathcal{B}\subseteq \mathcal{A}$. Since the coordinate functions
  $\eta_\beta: X\to \R$, $\beta \in \mathcal{B}$, are continuous, so
  is the map
  $\eta\uhr U:U\to \R^\mathcal{B}\subseteq
  \mathbf{c}_{00}(\mathcal{A})$.  However, $\R^\mathcal{B}$ has a
  unique linear topology\,---\,the Euclidean one. Therefore,
  $\eta\uhr U$ is continuous with respect to the finite topology on
  $\mathbf{c}_{00}(\mathcal{A})$. Thus,
  $\eta:X\to \mathbf{c}_{00}(\mathcal{A})$ is also continuous with
  respect to the finite topology on $\mathbf{c}_{00}(\mathcal{A})$.
\end{proof}

Complementary to \Cref{proposition-Simpl-Qst-v26:1} is the following
 observation for general vector spaces.  It is valid in a more general
 situation that is not directly relevant to the results of this paper.  

\begin{proposition}
  \label{proposition-CS-and-CM-15:1}
  Let $X$ be a space, $E$ be a vector space and
  $\lambda:X\to \mathbf{c}_{00}(\mathcal{A})\setminus\{\mathbf{0}\}$
  be a continuous locally finite-dimensional map for some
  $\mathcal{A}\subseteq E$.  Define $\eta:X\to E$ by
  \begin{equation}
    \label{eq:CS-and-CM-15:1}
    \eta(x)=\sum_{\alpha\in \mathcal{A}}\lambda_\alpha(x)\cdot
    \alpha= \sum_{\alpha\in \car_\lambda(x)}\lambda_\alpha(x)\cdot
    \alpha\quad \text{for every $x\in X$.}
  \end{equation}
  Then the map $\eta:X\to E$ is continuous with respect to the finite topology
  on $E$.
\end{proposition}

\begin{proof}
  By \Cref{proposition-CS-and-CM-20:1},
  $\car_\lambda:X\sto \mathcal{A}$ is a locally finite cover of $X$
  and therefore $\eta$ is well-defined. For the same reason, a point
  $p\in X$ is contained in an open subset ${U\subseteq X}$ such that
  ${\mathcal{B}=\car_\lambda[U]}$ is a finite subset of $\mathcal{A}$. Let
  $L\subseteq E$ be the linear subspace of $E$ spanned over the subset
  $\mathcal{B}\subseteq E$. Since $L$ is finite-dimensional, it
  follows from \cref{eq:CS-and-CM-15:1} that $\eta\uhr U:U\to L$ is
  continuous with respect to the Euclidean topology on
  $L$. Accordingly, $\eta:X\to E$ is also continuous with respect to
  the finite topology on $E$.
\end{proof}

Finally, let us also point out the following ``set-valued''
interpretation of the closure of the elements of a cover, which is not
directly related to functionally open covers, but will play an
interesting role in the next section. In this property,
$\Phi[A]=\bigcup_{x\in A}\Phi(x)$ is the \emph{image} of a subset
$A\subseteq X$ by a set-valued mapping $\Phi:X\sto Y$.  

\begin{proposition}
  \label{proposition-CS-and-CM-07:1}
  For a cover\/ $\Omega:X\sto \mathcal{A}$ of a space $X$, define a
  relation\/ $\overline{\Omega}^*\subseteq X\times \mathcal{A}$ by
  \begin{equation}
    \label{eq:CS-and-CM-07:1}
    \overline{\Omega}^{*}(p)=\bigcap\left\{\Omega[U]:
      U\subseteq X\ \text{is open and}\ p\in U\right\},\quad p\in X.   
  \end{equation}
  Then\, $\overline{\Omega}^*:X\sto \mathcal{A}$ is a cover of\/ $X$
  such that\/
  $\left(\overline{\Omega}^*\right)^{-1}(\alpha)=
  \overline{\Omega^{-1}(\alpha)}$\, for every $\alpha\in \mathcal{A}$.
\end{proposition}

\begin{proof}
  Take an element $\alpha\in \mathcal{A}$ and a point $p\in X$. Then
  $p\in \overline{\Omega^{-1}(\alpha)}$ if and only if
  ${U\cap \Omega^{-1}(\alpha)\neq \emptyset}$ for every open set
  $U\subseteq X$ with $p\in U$. However,
  $U\cap \Omega^{-1}(\alpha)\neq \emptyset$ precisely when
  $\alpha\in \Omega[U]$. According to \cref{eq:CS-and-CM-07:1}, this
  shows that
  $\left(\overline{\Omega}^*\right)^{-1}(\alpha)=
  \overline{\Omega^{-1}(\alpha)}$.
\end{proof}

We conclude this section with several remarks.

\begin{remark}
  \label{remark-CS-and-CM-20:1}
  A mapping $\Phi:X\sto Y$ between spaces $X$ and $Y$ is called
  \emph{lower locally constant} \cite{gutev:05} if the set
  $\{x\in X:K\subseteq \Phi(x)\}$ is open in $X$ for every compact
  subset $K\subseteq Y$. The defining property of lower locally
  constant mappings originates from a paper by Uspenskij
  \cite{uspenskij:98}. Later, these mappings were used by some authors
  (see, for example, \cite{chigogidze-valov:00a,valov:00}) under the
  name ``strongly l.s.c.'', while in papers by other authors,
  ``strongly l.s.c.'' meant a different property of set-valued
  mappings (see, for example, \cite{gutev:95e}). The lower locally
  constant mappings played a key role in several selection results,
  the interested reader is referred to
  \cite{gutev:2018a,Gutev2020b,uspenskij:98,valov:00}. Evidently,
  every lower locally constant mapping $\Phi:X\sto Y$ is
  totally-l.s.c. (i.e.\ an open cover of $X$), but the converse is not
  true. For instance, define $\Phi:[0,1)\sto [0,1]$ by $\Phi(0)=[0,1]$
  and $\Phi(t)=[0,1-t)\cup\{1\}$ for every $t\in (0,1)$. Then
  $\Phi^{-1}(0)=[0,1)=\Phi^{-1}(1)$ and $\Phi^{-1}(t)=[0,1-t)$ for
  $t\in(0,1)$. Accordingly, $\Phi:[0,1)\sto[0,1]$ is an open cover of
  $[0,1)$, but it is not lower locally constant because the set
  $\left\{t\in[0,1): [0,1]\subseteq \Phi(t)\right\}=\{0\}$ is not open
  in $[0,1)$. Similarly, there are lower locally constant mappings
  which are not open relations. One of the simplest examples is given
  in \cite{Gutev2023}, where $Y$ is a non-discrete space and $X$ is
  the set $Y$ endowed with the discrete topology. Then the identity
  relation $\Delta = \{\langle y,y\rangle: y \in Y\}$ is lower locally
  constant, but it is not an open relation in $X\times Y$.
\end{remark}

\begin{remark}
  \label{remark-CS-and-CM-15:1}
  A mapping $\Phi:X\sto Y$ between spaces $X$ and $Y$ is called
  \emph{subcontinuous} if for every open cover of $Y$, every point
  $p\in X$ is contained in an open set $U\subseteq X$ such that
  $\Phi[U]$ is covered by finitely many members of this cover. It
  should be remarked that this is not the original definition of
  subcontinuity, it is a natural characterisation stated in
  \cite[Theorem 7.1]{MR667078} and credited to
  \cite{lechicki:80}. Originally, subcontinuity was defined for usual
  maps $f:X\to Y$ in \cite{MR0227952}. Subsequently, the property was
  extended to set-valued mappings in \cite{MR0410656}. In these terms,
  for a discrete space $\mathcal{A}$, a cover
  $\Omega:X\sto \mathcal{A}$ of a space $X$ is locally finite exactly
  when it is a subcontinuous mapping. Similarly, in the same setting,
  it follows from \Cref{proposition-CS-and-CM-20:1} that a map
  $\eta:X\to \R^\mathcal{A}\setminus\{\mathbf{0}\}$ is locally
  finite-dimensional if and only if the relation
  $\car_\eta:X\sto \mathcal{A}$ is subcontinuous.
\end{remark}

\begin{remark}
  \label{remark-CS-and-CM-15:2}
  A mapping $\Phi :X\sto Y$ between spaces $X$ and $Y$ is \emph{upper
    semi-continu\-ous}, or \emph{u.s.c.}, if $\Phi^{-1}[F]$ is closed
  in $X$ for every closed $F\subseteq Y$. A compact-valued u.s.c.\
  mapping is usually called \emph{usco}. The relation corresponding to
  an usco mapping ${\Phi:X\sto Y}$ is always closed in $X\times Y$. It
  is also well known that each usco mapping is subcontinuous, see,
  e.g., \cite[Theorem 2.5]{MR500038}. In fact, a subcontinuous mapping
  is usco precisely when it has a closed graph. Namely, for
  $\Phi:X\sto Y$, let $\overline{\Phi}^{X\times Y}:X\sto Y$ be the
  mapping corresponding to the \emph{closure} of the relation
  $\Phi\subseteq X\times Y$. Then, as shown in \cite[Theorem
  3.3]{MR0275390},
  \begin{equation}
    \label{eq:CS-and-CM-07:2}
    \overline{\Phi}^{X\times Y}(p)=\bigcap\left\{\overline{\Phi[U]}:
      U\subseteq X\ \text{is open and}\ p\in U\right\}\quad \text{for
      every $p\in X$.}
  \end{equation}
  An immediate consequence of this representation is that
  $\Phi:X\sto Y$ is subcontinuous if and only if so is the mapping
  $\overline{\Phi}^{X\times Y}:X\sto Y$. Furthermore, it was shown in
  \cite[Theorem 3.1]{MR0410656} that each closed-graph subcontinuous
  mapping $\Phi:X\sto Y$ is usco. Thus, $\Phi:X\sto Y$ is
  subcontinuous exactly when the mapping
  $\overline{\Phi}^{X\times Y}:X\sto Y$ is usco. However, for a
  discrete space $\mathcal{A}$ and a cover $\Omega:X\sto \mathcal{A}$,
  it follows from \cref{eq:CS-and-CM-07:1,eq:CS-and-CM-07:2} that
  $\overline{\Omega}^*=\overline{\Omega}^{X\times\mathcal{A}}$. Accordingly,
  in this case, a cover $\Omega:X\sto \mathcal{A}$ is locally finite
  if and only if the mapping\, $\overline{\Omega}^*:X\sto \mathcal{A}$
  is usco.
\end{remark}

\section{A Universal Locally Finite Partition of Unity}
\label{sec:univ-locally-finite}

A \emph{shrinking} of an indexed cover
$\left\{U_\alpha:\alpha\in \mathcal{A}\right\}$ of a set $X$ is a
cover $\left\{V_\alpha:\alpha\in \mathcal{A}\right\}$ of $X$ such that
$V_\alpha\subseteq U_\alpha$ for every $\alpha\in \mathcal{A}$. In the
interpretation of covers as set-valued mappings, a shrinking is
equivalent to a \emph{multi-selection}, because a cover
$\Phi:X\sto \mathcal{A}$ of $X$ is a shrinking of
$\Omega:X\sto \mathcal{A}$ simply when $\Phi\subseteq \Omega$. Thus,
in these terms, a partition of unity
$\xi:X\to \mathfrak{S}(\mathcal{A})$ on a space $X$ is
\emph{index-subordinated} to a cover\/ $\Omega:X\sto \mathcal{A}$ if
${\car_\xi\subseteq \Omega}$. An alternative definition of
index-subordinated partitions of unity is discussed in
\Cref{remark-st-app-vgg-rev:1}, but it is essentially equivalent to
our interpretation.\medskip

A partition of unity ${\xi:X\to \mathfrak{S}(\mathcal{A})}$ on a space
$X$ is \emph{locally finite} if ${\car_\xi:X\sto \mathcal{A}}$ is a
locally finite cover of $X$. Equivalently, by
\Cref{proposition-CS-and-CM-20:1},
${\xi:X\to \mathfrak{S}(\mathcal{A})}$ is a locally finite partition
of unity on $X$ precisely when it is a continuous locally
finite-dimensional map. Furthermore, in this case,
$\xi:X\to \mathfrak{S}(\mathcal{A})$ takes values in the
\emph{positive cone} $\mathbf{c}_{00}^+(\mathcal{A})$ of
$\mathbf{c}_{00}(\mathcal{A})$, in fact in the following subset of
$\mathbf{c}_{00}^+(\mathcal{A})$:
\begin{equation}
  \label{eq:CS-and-CM-16:2}
  \boldsymbol{\Sigma}(\mathcal{A})= \mathfrak{S}(\mathcal{A})\cap
  \mathbf{c}_{00}(\mathcal{A})= \left\{y\in
    \mathbf{c}_{00}^+(\mathcal{A}): \|y\|_1=1\right\}. 
\end{equation}

A cover of a space has an index-subordinated locally finite partition
of unity whenever it has an index-subordinated partition of
unity. This property is well known and is based on a construction of
M. Mather, see \cite[Lemma]{MR0281155} and \cite[Lemma
5.1.8]{engelking:89}. We will now apply this construction to show that
this relationship between partitions of unity and locally finite
partitions of unity can be expressed only in terms of the identity map
$1_{\mathfrak{S}(\mathcal{A})}$ of the convex subset
$\mathfrak{S}(\mathcal{A})\subseteq \ell_1(\mathcal{A})$.

\begin{theorem}
  \label{theorem-CS-and-CM-15:1}
  For an infinite set $\mathcal{A}$, there exists a continuous locally
  finite-dimen\-sional map
  $\eta:\mathfrak{S}(\mathcal{A})\to \boldsymbol{\Sigma}(\mathcal{A})$
  such that\/ $\overline{\car_\eta}^*\subseteq \car$. Accordingly, if\/
  $\xi:X\to \mathfrak{S}(\mathcal{A})$ is a continuous map on a space
  $X$, then composite map
  $\eta\circ \xi:X\xrightarrow[]{\xi}
  \mathfrak{S}(\mathcal{A})\xrightarrow[]{\eta}
  \boldsymbol{\Sigma}(\mathcal{A})$ is a continuous locally
  finite-dimensional map such that\,
  $\overline{\car_{\eta\circ \xi}}^* \subseteq\car_\xi$.
\end{theorem}

\begin{proof}
  The \emph{sup-norm} $\|\cdot\|_\infty$ on $\ell_1(\mathcal{A})$ is
  defined by
  $ \|y\|_\infty= \sup_{\alpha \in \mathcal{A}} |y (\alpha)| $ for
  each $y\in \ell_1(\mathcal{A})$. Since
  $\|\cdot\|_\infty\leq \|\cdot\|_1$, it is a continuous function on
  the normed space
  $\left(\ell_1(\mathcal{A}),\|\cdot\|_1\right)$. Moreover, each
  projection
  $\pi_\alpha\uhr
  \mathfrak{S}(\mathcal{A}):\mathfrak{S}(\mathcal{A})\to \R$,
  $\alpha\in \mathcal{A}$, is also continuous and
  $\pi_\alpha(y)=y(\alpha)$ for every
  $y\in \mathfrak{S}(\mathcal{A})$. Using this, for each
  $\alpha\in \mathcal{A}$, define a continuous function
  $\lambda_\alpha:\mathfrak{S}(\mathcal{A})\to\R$ by
  \begin{equation}
    \label{eq:CS-and-CM-04:2}
    \lambda_\alpha(y)=
    \max\left\{y(\alpha)-\frac{\|y\|_\infty}2,0\right\}\quad \text{for every
      ${y\in \mathfrak{S}(\mathcal{A})}$.}
  \end{equation}
  Then the diagonal map
  $\lambda=\Delta_{\alpha\in
    \mathcal{A}}\lambda_\alpha:\mathfrak{S}(\mathcal{A})\to
  \R^\mathcal{A}$ takes values in the positive cone
  $\mathbf{c}_{00}^+(\mathcal{A})$ of $\mathbf{c}_{00}(\mathcal{A})$
  because $\|y\|_\infty>0$ for every
  $y\in\mathfrak{S}(\mathcal{A})$. Furthermore, the relation
  $\car_\lambda\subseteq \mathfrak{S}(\mathcal{A}) \times \mathcal{A}$
  represents a mapping
  $\car_\lambda: \mathfrak{S}(\mathcal{A})\sto \mathcal{A}$ with the
  property that
  \begin{equation}
    \label{eq:2}
    y(\alpha)\geq \frac{\|y\|_\infty}2> 0\quad \text{for every
      $\alpha\in \car_\lambda(y)$.}
  \end{equation}
  This not only implies that $\car_\lambda\subseteq \car$, but also that
  $\overline{\car_\lambda}^*\subseteq \car$. Indeed, let
  $\alpha\in \mathcal{A}$ and
  $\{y_n\}\subseteq \car_\lambda^{-1}(\alpha)$ be a sequence that
  converges to some $y\in \ell_1^+(\mathcal{A})$. Then by
  \cref{eq:2},
  \[
    y(\alpha)=\lim_{n\to\infty}y_n(\alpha)\geq
    \lim_{n\to\infty}\frac{\|y_n\|_\infty}2=\frac{\|y\|_\infty}2>0.
  \]
  Thus, $\alpha\in\car(y)$ and by \Cref{proposition-CS-and-CM-07:1},
  $\overline{\car_\lambda}^*\subseteq \car$.\medskip

  We complete the proof by first showing that the map
  $\lambda:\mathfrak{S}(\mathcal{A})\to
  \mathbf{c}_{00}^+(\mathcal{A})\setminus\{\mathbf{0}\}$ is both
  continuous and locally finite-dimensional. So, take a point
  $p\in \mathfrak{S}(\mathcal{A})$. Since
  ${\|p\|_1=\sum_{\alpha\in \mathcal{A}}p(\alpha)}$, there exists a
  finite set $\mathcal{B}\subseteq \mathcal{A}$ with
  $\|p\|_1-\sum_{\beta\in
    \mathcal{B}}p(\beta)<\frac{\|p\|_\infty}2$. However, the norms
  $\|\cdot\|_\infty$ and $\|\cdot\|_1$ are continuous on
  $\ell_1(\mathcal{A})$, as are the projections $\pi_\beta$,
  $\beta\in \mathcal{B}$. Hence, the point
  $p\in \mathfrak{S}(\mathcal{A})$ is contained in an open set
  $U\subseteq \mathfrak{S}(\mathcal{A})$ such that
  \[
    \sum_{\alpha\in \mathcal{A}\setminus\mathcal{B}} y(\alpha)=
    \|y\|_1-\sum_{\beta\in \mathcal{B}}
    y(\beta)<\tfrac{\|y\|_\infty}2\quad \text{for every $y\in U$.}
  \]
  Thus, according to \cref{eq:2},
  $\car_\lambda[U]\subseteq \mathcal{B}$ which shows that
  $\car_\lambda:\mathfrak{S}(\mathcal{A})\sto \mathcal{A}$ is a
  locally finite cover $\mathfrak{S}(\mathcal{A})$ or, in other words,
  that the map
  $\lambda:\mathfrak{S}(\mathcal{A})\to
  \mathbf{c}_{00}^+(\mathcal{A})\setminus\{\mathbf{0}\}$ is locally
  finite-dimensional (see
  \Cref{proposition-CS-and-CM-20:1}). Similarly, according to
  \cref{eq:CS-and-CM-04:2}, this also implies that
  $\lambda:\mathfrak{S}(\mathcal{A})\to
  \mathbf{c}_{00}^+(\mathcal{A})\setminus\{\mathbf{0}\}$ is continuous
  (see \Cref{proposition-Simpl-Qst-v26:1}). Finally, the
  map
  $\eta:\mathfrak{S}(\mathcal{A})\to
  \boldsymbol{\Sigma}(\mathcal{A})$, defined by
  $\eta[y]=\frac{\lambda[y]}{\|\lambda[y]\|_1}$,
  $y\in \mathfrak{S}(\mathcal{A})$, is as required. Indeed,
  $\car_\eta=\car_\lambda$ and for a partition of unity
  $\xi:X\to \mathfrak{S}(\mathcal{A})$ on a space $X$, it follows from
  \cref{eq:Simpl-Qst-v11:1,eq:CS-and-CM-07:1} that
  \[
    \overline{\car_{\eta\circ \xi}}^*(p)\subseteq
    \overline{\car_\eta}^*(\xi[p]) \subseteq \car(\xi[p])=
    \car_\xi(p)\quad \text{for every $p\in X$.}\qedhere
  \]
\end{proof}

\begin{remark}
  \label{remark-st-app-vgg-rev:1}
  In some sources, a partition of unity
  $\xi:X\to \mathfrak{S}(\mathcal{A})$ on a space $X$ is called
  \emph{index-subordinated} to a cover $\Omega:X\sto \in\mathcal{A}$
  of $X$ if
  $\supp(\xi_\alpha)=\overline{\coz(\xi_\alpha)}\subseteq
  \Omega^{-1}(\alpha)$ for every $\alpha\in \mathcal{A}$;
  equivalently, by \Cref{proposition-CS-and-CM-07:1}, if
  $\overline{\car_\xi}^*\subseteq \Omega$. However, these variations
  in the terminology do not affect the results of this
  section. Namely, according to \Cref{theorem-CS-and-CM-15:1}, if
  $\xi:X\to \mathfrak{S}(\mathcal{A})$ is a partition of unity on a
  space $X$, then $X$ also has a (locally finite) partition of unity
  $\gamma:X\to \boldsymbol{\Sigma}(\mathcal{A})$ such that
  $\overline{\car_\gamma}^* \subseteq\car_\xi$, i.e.\ with the
  property that $\supp(\gamma_\alpha)\subseteq \coz(\xi_\alpha)$ for
  all $\alpha\in \mathcal{A}$. This is essentially M. Mather's
  construction and was also explicitly stated in \cite[Proposition
  2.7.4]{MR3099433}.
\end{remark}

\section{Selections and Partitions of Unity}
\label{sec:select-part-unity}

A set $\mathcal{A}$ can be considered as a subset of
$\mathbf{c}_{00}(\mathcal{A})$ when we identify each
$\alpha\in \mathcal{A}$ with its characteristic function
$\alpha:\mathcal{A}\to \{0,1\}$, defined by $\alpha(\beta)=0$ for
${\alpha\neq \beta}$ and $\alpha(\alpha)=1$. Thus, in fact,
$\mathcal{A}$ becomes a Hamel basis for $\mathbf{c}_{00}(\mathcal{A})$
with the property that
$\mathcal{A}\subseteq\boldsymbol{\Sigma}(\mathcal{A})$, see
\cref{eq:CS-and-CM-16:2}. Hence, to each cover
$\Omega:X\sto \mathcal{A}$ of a set $X$ we may associate the mapping
$\conv[\Omega]:X\sto \boldsymbol{\Sigma}(\mathcal{A})$ defined by the
\emph{convex hull} of the point images of $\Omega$, i.e.\ by
$\conv[\Omega](x)=\conv(\Omega(x))$ for every $x\in X$.  Since each
$y\in \boldsymbol{\Sigma}(\mathcal{A})$ is both a function
$y:\mathcal{A}\to \R$ with a finite support
$\car(y)\subseteq \mathcal{A}$ and a unique convex combination of the
elements of $\car(y)$, the representation
$y= \sum_{\alpha\in \car(y)} y(\alpha)\cdot \alpha$ is also unique.
This implies the following important properly of the associated
mapping $\conv[\Omega]$.
\begin{equation}
  \label{eq:CS-and-CM-16:3}
  \car(y)\subseteq
  \Omega(x)\quad\text{for every $y\in \conv[\Omega](x)$ and $x\in X$.}
\end{equation}
Furthermore, $\conv[\Omega]:X\sto \boldsymbol{\Sigma}(\mathcal{A})$ is
an open cover of a space $X$ whenever $\Omega:X\sto \mathcal{A}$ is
such. In short, for $y\in \conv[\Omega](x)$, it follows from
\cref{eq:CS-and-CM-16:3} that ${\car(y)\subseteq \Omega(x)}$ and hence
$\left(\conv[\Omega]\right)^{-1}(y)=\bigcap_{\alpha\in
  \car(y)}\Omega^{-1}(\alpha)$ is an open set containing the point
$x\in X$. Thus, in terms of totally-l.s.c.\ mappings, we have the
following useful observation.

\begin{proposition}
  \label{proposition-CS-and-CM-17:1}
  If\/ $\Omega:X\sto \mathcal{A}$ is a totally-l.s.c.\ mapping, then
  the associated mapping\/
  $\conv[\Omega]:X\sto \boldsymbol{\Sigma}(\mathcal{A})$ is also
  totally-l.s.c.
\end{proposition}

Based on this property, we now have the following result, which is a
slight generalisation of \Cref{theorem-CS-and-CM-19:5}. 

\begin{theorem}
  \label{theorem-CS-and-CM-16:1}
  For a space $X$, the following conditions are equivalent\textup{:}
  \begin{enumerate}
  \item\label{item:CS-and-CM-16:1} If\/ $E$ is a topological vector
    space, then each convex-valued totally-l.s.c.\ mapping\/
    $\Phi:X\sto E$ has a continuous selection.
  \item\label{item:CS-and-CM-16:2} If\/ $\mathcal{A}$ is a set, then
    each convex-valued totally-l.s.c.\ mapping\/
    $\Phi:X\sto \mathbf{c}_{00}(\mathcal{A})$ has a continuous
    selection.
  \item\label{item:CS-and-CM-16:3} Each open cover of\/ $X$ has an
    index-subordinated partition of unity.
  \item\label{item:CS-and-CM-17:3} If\/ $E$ is a vector space, then
    each convex-valued totally-l.s.c.\ mapping\/ $\Phi:X\sto E$ has a
    selection which is continuous with respect to the finite topology
    on $E$.
  \end{enumerate}
\end{theorem}

\begin{proof}
  The implication
  \ref{item:CS-and-CM-16:1}$\implies$\ref{item:CS-and-CM-16:2} is
  trivial. Assume that \ref{item:CS-and-CM-16:2} is valid,
  $\Omega:X\sto \mathcal{A}$ is an open cover of $X$, and
  $\conv[\Omega]:X\sto \boldsymbol{\Sigma}(\mathcal{A})$ is the
  convex-valued mapping associated to $\Omega$. Then by
  \ref{item:CS-and-CM-16:2} and \Cref{proposition-CS-and-CM-17:1},
  $\conv[\Omega]$ has a continuous selection
  ${\xi:X\to \boldsymbol{\Sigma}(\mathcal{A})}$. Accordingly, $\xi$ is
  a partition of unity on $X$. Moreover, by \cref{eq:CS-and-CM-16:3},
  $\xi$ is also index-subordinated to the cover
  $\Omega:X\sto \mathcal{A}$ of $X$. This shows that
  \ref{item:CS-and-CM-16:2}$\implies$\ref{item:CS-and-CM-16:3}.\medskip
  
  To show that
  \ref{item:CS-and-CM-16:3}$\implies$\ref{item:CS-and-CM-17:3}, assume
  that $E$ is a vector space, ${\Phi:X\sto E}$ is a convex-valued
  totally-l.s.c.\ mapping and set $\mathcal{A}=\Phi[X]$. Then by
  \ref{item:CS-and-CM-16:3}, \Cref{proposition-CS-and-CM-20:1} and
  \Cref{theorem-CS-and-CM-15:1}, the open cover
  $\Phi:X\sto \mathcal{A}$ of $X$ has a locally finite
  index-subordinated partition of unity
  $\gamma:X\to \boldsymbol{\Sigma}(\mathcal{A})$. Next, define a map
  $\varphi:X\to E$ as in \cref{eq:CS-and-CM-15:1}, i.e.\ by
  $\varphi(x)=\sum_{\alpha\in \car_\gamma(x)}\gamma_\alpha(x)\cdot
  \alpha$ for every $x\in X$. Evidently, $\varphi$ is a selection for
  $\Phi:X\sto \mathcal{A}$ because $\Phi$ is convex-valued and
  $\gamma_\alpha(x)\in \Phi(x)$ for every $x\in X$ and
  $\alpha\in \car_\gamma(x)$. Moreover, by
  \Cref{proposition-CS-and-CM-15:1}, $\varphi$ is continuous with
  respect to the finite topology on $E$ because
  $\gamma:X\to \boldsymbol{\Sigma}(\mathcal{A})\subseteq
  \R^\mathcal{A}\setminus\{\mathbf{0}\}$ is continuous and locally
  finite-dimensional. Thus, $\varphi$ is as required in
  \ref{item:CS-and-CM-17:3}. Since the implication
  \ref{item:CS-and-CM-17:3}$\implies$\ref{item:CS-and-CM-16:1} is
  trivial, the proof is complete.
\end{proof}

For a metric space $(Y,\rho)$ and $\varepsilon>0$, let
$\mathbf{O}_\varepsilon(p)=\left\{y\in Y: d(y,p)<\varepsilon\right\}$
be the \emph{open $\varepsilon$-ball centred at a point $p\in Y$}. In
this setting, to each $\Phi:X\sto Y$ we can associate the mapping
$\mathbf{O}_\varepsilon[\Phi]:X\sto Y$ defined by
$\mathbf{O}_\varepsilon[\Phi](x)=\mathbf{O}_\varepsilon\left(\Phi
  (x)\right)= \bigcup_{y\in \Phi(x)}\mathbf{O}_\varepsilon(y)$ for
every $x\in X$. This is naturally related to approximate
selections. Namely, a map $f:X\to Y$ is an
\emph{$\varepsilon$-selection} for $\Phi :X\sto Y$ if $f$ is a
selection for $\mathbf{O}_\varepsilon[\Phi]:X\sto Y$, i.e.\ if
$f(x)\in \mathbf{O}_\varepsilon\left(\Phi (x)\right)$ for every
$x\in X$. The following general observation was summarised
\cite[Proposition 2.12]{Gutev2020a}, it is a partial case of
\cite[Proof that Lemma 5.1 implies Theorem 4.1, page
569]{michael:56b}.

\begin{proposition}[\cite{michael:56b}]
  \label{proposition-shsa-vgg-rev:1}
  For a space $X$ and a Banach space $E$, the following conditions are
  equivalent\textup{:}
  \begin{enumerate}
  \item\label{item:shsa-vgg-rev:2} Each closed-convex-valued l.s.c.\
    mapping\/ $\Phi:X\sto E$ has a continuous selection.
  \item\label{item:shsa-vgg-rev:3} Each convex-valued l.s.c.\
    mapping\/ $\Phi:X\sto E$ has a continuous $\varepsilon$-selection,
    for every $\varepsilon>0$.
  \end{enumerate}
\end{proposition}

We now have the following result, which also simplifies the proof of
\Cref{theorem-CS-and-CM-19:1}. 

\begin{theorem}
  \label{theorem-CS-and-CM-16:2}
  For a space $X$ and an infinite set $\mathcal{A}$, the following
  conditions are equivalent\textup{:}
  \begin{enumerate}
  \item\label{item:CS-and-CM-16:4} Each convex-valued l.s.c.\
    mapping\/ $\Phi:X\sto \mathbf{c}_{00}(\mathcal{A})$ has a
    continuous $\varepsilon$-selec\-tion, for every $\varepsilon>0$.
  \item\label{item:CS-and-CM-16:5} Each convex-valued l.s.c.\
    mapping\/ $\Phi:X\sto \ell_1(\mathcal{A})$ has a continuous
    $\varepsilon$-selec\-tion, for every $\varepsilon>0$.
  \item\label{item:CS-and-CM-16:6} Each open cover\/
    $\Omega:X\sto \mathcal{A}$ of $X$ has an index-subordinated
    partition of unity.
  \end{enumerate}
\end{theorem}

\begin{proof}
  To show that
  \ref{item:CS-and-CM-16:4}$\implies$\ref{item:CS-and-CM-16:5}, let
  $\Phi:X\sto \ell_1(\mathcal{A})$ be a convex-valued l.s.c.\ mapping
  and ${\varepsilon>0}$. Then the associated mapping
  $\mathbf{O}_{\varepsilon/2}[\Phi]:X\sto \ell_1(\mathcal{A})$ remains
  convex-valued and l.s.c.\ being an open-graph mapping, see
  \cite[Proposition 2.1]{gutev:05} and \cite[Proposition
  3.2]{Gutev2023}. Since $\mathbf{c}_{00}(\mathcal{A})$ is a dense
  convex subset of $\ell_1(\mathcal{A})$, we can define another
  convex-valued l.s.c.\ mapping
  $\Psi:X\sto \mathbf{c}_{00}(\mathcal{A})$ by
  $\Psi(x)=\mathbf{O}_{\varepsilon/2}[\Phi](x)\cap
  \mathbf{c}_{00}(\mathcal{A})$ for every $x\in X$, see
  e.g. \cite[Proposition 2.3]{michael:56a}. Finally, by
  \ref{item:CS-and-CM-16:4}, $\Psi$ has a continuous
  $\frac{\varepsilon}2$-selection
  $\psi:X\to \mathbf{c}_{00}(\mathcal{A})$. Evidently, $\psi$ is a
  continuous $\varepsilon$-selection for $\Phi$, i.e.\ as required in
  \ref{item:CS-and-CM-16:5}.\medskip

  Assume that \ref{item:CS-and-CM-16:5} holds,
  $\Omega:X\sto \mathcal{A}$ is an open cover of $X$, and
  $\conv[\Omega]:X\sto \boldsymbol{\Sigma}(\mathcal{A})$ is the
  convex-valued mapping associated to $\Omega$. Next, using
  \Cref{proposition-CS-and-CM-17:1} and \cite[Proposition
  2.3]{michael:56a}, we can define a closed-convex-valued l.s.c.\
  mapping ${\Phi:X\sto \mathfrak{S}(\mathcal{A})}$ by
  $\Phi(x)=\overline{\conv[\Omega](x)}$ for every ${x\in X}$. If
  $\alpha\in \car(y)$ for some $y\in \Phi(x)$ and $x\in X$, then
  $y(\alpha)>0$ and accordingly there is $p\in \conv[\Omega](x)$ with
  $\|p-y\|_1<y(\alpha)$. This implies that $p(\alpha)>0$. Hence, it
  follows from \cref{eq:CS-and-CM-16:3} that $\alpha\in \Omega(x)$
  and, therefore,
  \begin{equation}
    \label{eq:CS-and-CM-17:2}
    \car(y)\subseteq
    \Omega(x)\quad\text{for every $y\in \Phi(x)$ and $x\in X$.}
  \end{equation}
  Finally, by \ref{item:CS-and-CM-16:5} and
  \Cref{proposition-shsa-vgg-rev:1}, $\Phi$ has a continuous selection
  $\varphi:X\to \mathfrak{S}(\mathcal{A})$. Thus, by
  \cref{eq:CS-and-CM-17:2}, $\varphi$ is a partition of unity on $X$
  that is index-subordinated to the cover $\Omega:X\sto
  \mathcal{A}$. This shows that
  \ref{item:CS-and-CM-16:5}$\implies$\ref{item:CS-and-CM-16:6}.\medskip

  To see the last implication, assume that
  $\Phi:X\sto \mathbf{c}_{00}(\mathcal{A})$ is a convex-valued l.s.c.\
  mapping and $\varepsilon>0$. Then the relation
  $\mathbf{O}_\varepsilon[\Phi]\subseteq X\times
  \mathbf{c}_{00}(\mathcal{A})$ is open, see \cite[Proposition
  2.1]{gutev:05} and \cite[Proposition 3.1]{Gutev2023}. Next, take a
  dense subset $A\subseteq \mathbf{c}_{00}(\mathcal{A})$ with
  $|A|= |\mathcal{A}|$, and define an open cover $\Omega:X\sto A$ of
  $X$ by $\Omega(x)=\mathbf{O}_\varepsilon[\Phi](x)\cap A$ for every
  $x\in X$. It now follows from \ref{item:CS-and-CM-16:6} that
  $\Omega$ has an index-subordinated partition of unity. Hence, by
  \Cref{theorem-CS-and-CM-15:1}, there exists a continuous locally
  finite-dimensional map $\gamma:X\to \boldsymbol{\Sigma}(A)$ such
  that
  ${\car_\gamma(x)\subseteq \Omega(x)\subseteq
    \mathbf{O}_\varepsilon[\Phi](x)}$ for every $x\in X$. Finally, as
  in \cref{eq:CS-and-CM-15:1}, we can define a map
  $\varphi:X\to \mathbf{c}_{00}(\mathcal{A})$ by
  $\varphi[x]=\sum_{a\in \car_\gamma(x)} \gamma_a(x)\cdot a$, $x\in
  X$. Then by \Cref{proposition-Simpl-Qst-v26:1}, $\varphi$ is
  continuous. Moreover, it is a selection for
  $\mathbf{O}_\varepsilon[\Phi]$ because
  $\mathbf{O}_\varepsilon[\Phi]$ is convex-valued. Thus, $\varphi$ is
  a continuous $\varepsilon$-selection for $\Phi$ as required in
  \ref{item:CS-and-CM-16:4}.
\end{proof}

\begin{remark}
  \label{remark-CS-and-CM-17:1}
  An open cover $\Omega:X\sto Y$ of a space $X$ is independent of any
  topology on $Y$. In contrast, for a metric space $(Y,d)$, a
  continuous $\varepsilon$-selection for a mapping $\Omega:X\sto Y$ is
  naturally related to the metric structure on $Y$. This may explain
  the difference between \Cref{theorem-CS-and-CM-16:1} and
  \Cref{theorem-CS-and-CM-16:2}. Namely, in
  \Cref{theorem-CS-and-CM-16:1} the conditions are for any set
  $\mathcal{A}$, while \Cref{theorem-CS-and-CM-16:2} is valid for a
  fixed set $\mathcal{A}$.
\end{remark}

\section{Canonical Maps and Selections}
\label{sec:canon-maps-select}

The so-called canonical maps are closely related to partitions of
unity, but unlike them, they are essentially defined by simplicial
complexes. Another difference is that it will be more convenient to
consider these maps in terms of unindexed covers of a given space,
i.e.\ in this section a cover of a space $X$ is simply a family
$\mathcal{U}$ of subsets of $X$ such that
$\bigcup \mathcal{U}=X$.\medskip

A \emph{simplicial complex} is a collection $\Sigma$ of nonempty
finite subsets of a set $V$ such that $\tau\in \Sigma$, whenever
$\emptyset\neq \tau\subseteq \sigma\in \Sigma$. Each element of
$\Sigma$ is called a \emph{simplex}. The subset
$V(\Sigma)=\bigcup\Sigma \subseteq V$ is called the \emph{vertex set}
of $\Sigma$.  We can always assume that $V(\Sigma)=V$ and in the
sequel we will denote this set simply by $V$. As in the previous
section, identifying each $v\in V$ with its characteristic function
$v:V\to \{0,1\}$, the vertex set ${V}$ of a simplicial complex
$\Sigma$ is a linearly independent subset of the vector space
$\mathbf{c}_{00}\left({V}\right)$.  Then to each $\sigma\in \Sigma$ we
can associate the \emph{geometric simplex} $|\sigma|=\conv(\sigma)$,
which is the convex hull of $\sigma$. The resulting set
$|\Sigma|=\bigcup_{\sigma\in
  \Sigma}|\sigma|\subseteq\mathbf{c}_{00}^+\left({V}\right)$ is called
the \emph{geometric realisation} of $\Sigma$. As a topological space,
we will consider $|\Sigma|$ endowed with the \emph{Whitehead topology}
\cite{MR1576810,MR0030759}. In this topology, a subset
$U\subseteq |\Sigma|$ is open if and only if $U\cap |\sigma|$ is open
in $|\sigma|$ for every $\sigma\in \Sigma$.  In fact, the Whitehead
topology on $|\Sigma|$ is the subspace topology with respect to the
finite topology on the vector space $\mathbf{c}_{00}({V})$. Thus, the
Whitehead topology on $|\Sigma|$ is not necessarily the subspace
topology on $|\Sigma|$ as a subset of the normed space
$\left(\mathbf{c}_{00}({V}),\|\cdot\|_1\right)$, but both topologies
coincide on each geometric simplex $|\sigma|$, for $\sigma\in
\Sigma$. In this section, we will use $|\Sigma|$ for the geometric
realisation of a simplicial complex $\Sigma$ equipped with the
Whitehead topology. When clarity requires it, we will also use
$|\Sigma|_m$ to express that $|\Sigma|$ is equipped with the
$\|\cdot\|_1$-norm on $\mathbf{c}_{00}({V})$.\medskip

We will now consider two examples of ``extreme''-like simplicial
complexes that are naturally related to partitions of unity. The first
example is the set
\begin{equation}
  \label{eq:CS-and-CM-22:1}
  \Sigma(\mathcal{A})=\left\{\sigma\subseteq \mathcal{A}: \sigma\neq
    \emptyset\ \text{and $\sigma$ is finite}\right\}. 
\end{equation}
Evidently, $\Sigma(\mathcal{A})$ is the maximal simplicial complex
whose set of vertices is $\mathcal{A}$ and its geometric realisation
is precisely the (convex) subset
$\boldsymbol{\Sigma}(\mathcal{A}) \subseteq
\mathbf{c}_{00}(\mathcal{A})$, see
\cref{eq:CS-and-CM-16:2}. Accordingly, the space
$|\Sigma(\mathcal{A})|_m$ is identical to the subspace
$\boldsymbol{\Sigma}(\mathcal{A})$ of the normed space
$\left(\mathbf{c}_{00}(\mathcal{A}),\|\cdot\|_1\right)$. In contrast,
$|\Sigma(\mathcal{A})|$ is still the same set
$\boldsymbol{\Sigma}(\mathcal{A})$, but now equipped with the
Whitehead topology. In other words, $|\Sigma(\mathcal{A})|$ is the set
$\boldsymbol{\Sigma}(\mathcal{A})$ considered as a subspace of the
vector space $\mathbf{c}_{00}(\mathcal{A})$ with respect to the finite
topology.\medskip

Another natural example of a simplicial complex is the \emph{nerve}
$\mathcal{N}(\mathcal{U})$ of a cover $\mathcal{U}$ of a set $X$,
which is the subcomplex of $\Sigma(\mathcal{U})$ defined by
\begin{equation}
  \label{eq:st-app-vgg-rev:8}
  \mathcal{N}(\mathcal{U})= \left\{\sigma\in
    \Sigma(\mathcal{U}):\bigcap\sigma\neq\emptyset\right\}.  
\end{equation}
The vertex set of $\mathcal{N}(\mathcal{U})$ is actually $\mathcal{U}$
because we can always assume that
$\emptyset\notin \mathcal{U}$.\medskip

Nerves of covers are naturally related to special partitions of unity,
and are defined using the open star, see
\cref{eq:Simpl-Qst-v16:2}. Namely, for a cover $\mathcal{U}$ of a
space $X$, a continuous map $f:X\to |\mathcal{N}(\mathcal{U})|$ is
called \emph{canonical for\/ $\mathcal{U}$} if
\begin{equation}
  \label{eq:st-app-v1:3}
  \st_f(U)=\coz\left(f_U\right)\subseteq U\quad \text{for every $U\in
    \mathcal{U}$.} 
\end{equation}
Since
$|\mathcal{N}(\mathcal{U})|\subseteq \mathbf{c}_{00}(\mathcal{U})$ and
$f:X\to |\mathcal{N}(\mathcal{U})|$ is continuous with respect to the
finite topology on $\mathbf{c}_{00}(\mathcal{U})$, each coordinate
function $f_U:X\to \R$, $U\in \mathcal{U}$, is also
continuous. Moreover,
$|\mathcal{N}(\mathcal{U})|\subseteq
\boldsymbol{\Sigma}(\mathcal{U})\subseteq \mathfrak{S}(\mathcal{U})$
and by \cref{eq:st-app-v1:3}, $\coz(f_U)\subseteq U$ for every
$U\in \mathcal{U}$. Thus, in particular, $f$ is a point-finite
partition on $X$ which is index-subordinated to the cover
$\mathcal{U}$. This implies the following property which is equivalent
to the conditions in \Cref{theorem-CS-and-CM-16:1}, see also
\cite[Theorem 2.4]{gutev:2018a}.

\begin{theorem}
  \label{theorem-st-app-v12:1}
  For a space $X$, the following conditions are equivalent\textup{:}
  \begin{enumerate}
  \item\label{item:CS-and-CM-22:1} If\/ $E$ is vector space, then each
    convex-valued totally-l.s.c.\ mapping\/ $\Phi:X\sto E$ has a
    selection which is continuous with respect to the finite topology
    on $E$.
  \item\label{item:CS-and-CM-22:2} Each open cover of\/ $X$ has a
    canonical map.
  \item\label{item:CS-and-CM-22:3} Each open cover of\/ $X$ has an
    index-subordinated partition of unity.
  \end{enumerate}
\end{theorem}

\begin{proof}
  To see that
  \ref{item:CS-and-CM-22:1}$\implies$\ref{item:CS-and-CM-22:2}, let
  $\mathcal{U}$ be an open cover of $X$. We can view this cover as a
  set-valued mapping indexed by itself, i.e.\ as the mapping
  $\mathcal{U}:X\sto \mathcal{U}$ defined by
  $\mathcal{U}(x)=\{U\in \mathcal{U}:x\in U\}$. Evidently, in this
  interpretation, $\mathcal{U}^{-1}(U)=U$ for every
  $U\in \mathcal{U}$. Next, using
  \cref{eq:CS-and-CM-22:1,eq:st-app-vgg-rev:8}, we can define another
  (open) cover $\Omega:X\sto \mathcal{N}(\mathcal{U})$ of $X$ by
  $\Omega(x)=\Sigma(\mathcal{U}(x))$ for every $x\in X$. Finally,
  define a mapping $\Phi:X\sto |\mathcal{N}(\mathcal{U})|$ by
  $\Phi(x)=|\Sigma(\mathcal{U}(x))|$ for every $x\in X$. Then
  $\Phi:X\sto\mathbf{c}_{00}(\mathcal{U})$ is convex-valued because
  each $\Phi(x)$, $x\in X$, is actually the convex subset
  $\boldsymbol{\Sigma}(\mathcal{U}(x))\subseteq
  \mathbf{c}_{00}(\mathcal{U})$. Moreover, $\Phi$ is a totally-l.s.c.\
  mapping such that
  \begin{equation}
    \label{eq:CS-and-CM-22:2}
    \car(p)\subseteq \mathcal{U}(x)\quad \text{for every 
      $p\in \Phi(x)$ and $x\in X$.}
  \end{equation}
  Indeed, if $p\in \Phi(x)=|\Sigma(\mathcal{U}(x))|$, then
  $p\in |\sigma|$ for some $\sigma\subseteq \mathcal{U}(x)$ with
  $\sigma\in \mathcal{N}(\mathcal{U})$. Since $\car(p)$ is the minimal
  simplex of $\mathcal{N}(\mathcal{U})$ with the property that
  $p\in |\car(p)|$, this implies that
  $\car(p)\subseteq \sigma\subseteq \mathcal{U}(x)$, which is
  \cref{eq:CS-and-CM-22:2}. This also implies that $\Phi$ is an open
  cover of $X$ because by \cref{eq:CS-and-CM-22:2}, $p\in \Phi(x)$ if
  and only if $\car(p)\subseteq \mathcal{U}(x)$. Accordingly,
  $\Phi^{-1}(p)=\bigcap\car(p)$ is an open subset of $X$.\medskip 

  Now we complete this implication as follows. By
  \ref{item:CS-and-CM-22:1}, the so constructed convex-valued
  totally-l.s.c.\ mapping
  $\Phi:X\sto |\mathcal{N}(\mathcal{U})|\subseteq
  \mathbf{c}_{00}(\mathcal{U})$ has a selection
  $f:X\to |\mathcal{N}(\mathcal{U})|$ that is continuous with respect
  to the finite topology on $\mathbf{c}_{00}(\mathcal{U})$. Hence, $f$
  is also continuous with respect to the Whitehead topology on the
  geometric realisation $|\mathcal{N}(\mathcal{U})|$ of the nerve of
  the cover $\mathcal{U}$ of $X$. Moreover, for $U\in \mathcal{U}$ and
  $x\in \st_f(U)=\coz\left(f_U\right)$, it follows from
  \cref{eq:Simpl-Qst-v11:1,eq:Simpl-Qst-v16:2} that $U\in
  \car(f(x))$. Finally, by \cref{eq:CS-and-CM-22:2}, $x\in U$ and,
  therefore, $\st_f(U)\subseteq U$. Thus, by \cref{eq:st-app-v1:3},
  $f$ is canonical for $\mathcal{U}$.\medskip

  The implication
  \ref{item:CS-and-CM-22:2}$\implies$\ref{item:CS-and-CM-22:3} is
  trivial because the metric topology on the geometric realisation
  $|\mathcal{N}(\mathcal{U})|$ of a nerve of a cover of $X$ is weaker
  than the Whitehead one. Since the implication
  \ref{item:CS-and-CM-22:3}$\implies$\ref{item:CS-and-CM-22:1} is a
  part of \Cref{theorem-CS-and-CM-16:1}, the proof is complete.
\end{proof}

Regarding several other properties of canonical maps, the interested
reader is referred to \cite{gutev:2018a}.

%%%\bibliographystyle{amsplain-ab}
%%%\bibliography{gutev}

\end{document}